\documentclass[a4paper, 12pt]{amsart}
\usepackage{geometry}
\usepackage{hyperref}
\usepackage{mathrsfs}                    
\usepackage{mathabx}              
\usepackage{setspace}
\usepackage{graphicx}
\usepackage{amssymb}
\usepackage{epstopdf}
\usepackage{fullpage}
\usepackage[all]{xy}
\usepackage{xcolor}
\newtheorem{theorem}{Theorem}[section]
\newtheorem{lemma}[theorem]{Lemma}
\newtheorem{corollary}[theorem]{Corollary}
\newtheorem{proposition}[theorem]{Proposition}

\newtheorem*{theorem*}{Theorem}

\theoremstyle{definition}
\newtheorem{definition}[theorem]{Definition}

\newtheorem{example}[theorem]{Example}

\theoremstyle{remark}
\newtheorem{remark}[theorem]{Remark}

\begin{document}
\title{Mutations and derived equivalences\\ for commutative noetherian rings}

\author{Jorge Vitória}


\address{Jorge Vit\'oria, Dipartimento di Matematica ''Tullio Levi-Civita``, Universit\`a degli Studi di Padova, via Trieste 63, 35131 Padova, Italy}
\email{jorge.vitoria@unipd.it}



\begin{abstract}
The representation theory of a commutative noetherian ring is tightly controlled by its prime spectrum. In this article we use the prime spectrum to describe mutation of cosilting objects in the derived category of a commutative noetherian ring. 
\end{abstract}
\maketitle

\section{Introduction}
Arguably the most popular use of the word ``mutation'' in science is in genetics. The same word has been used in algebra since the 80s for a process that, figuratively, comes close to that used in biology. A DNA mutation is a process that turns a DNA sequence into another one through various changes, three of which are dubbed \textit{deletion}, \textit{substitution} and \textit{insertion}. In algebra, mutation often stands for a transformation of an algebraic object (exceptional collections, cluster variables, cluster-tilting objects, silting or cositling complexes, support $\tau$-tilting modules, etc...), encompassing deletion, substitution and insertion in a sequential way in order to produce a new object of the same nature.

Our focus is on the recently developed concept of mutation for the class of bounded cosilting complexes in the derived category of a ring. These complexes parametrise an important class of t-structures, and it is through this lens that the operation of mutation is conceptually better understood. Nevertheless, it is often difficult to describe mutations explicitly. This is due to the ``substitution'' step of the process: it relies on approximation theory, and it may be challenging to pin down such approximations. This can be overcome for some particular rings. For finite-dimensional algebras over a field there is a wide range of combinatorial and homological techniques to help with this task (see, for example, \cite{AI,AIR,ALSV}). In this article, after reviewing the necessary tools to understand mutation conceptually, we aim to describe this process as accurately as possible in the setting of the derived category of a commutative noetherian ring. In this setting, we do not have the combinatorial techniques of finite-dimensional algebras, but we have a topological tool instead: the prime spectrum. 

It is well-known that the prime spectrum of a commutative noetherian ring $R$ controls much of its representation theory via the notion of support. Celebrated results include Gabriel's classification of hereditary torsion classes of $\mathsf{Mod}(R)$ (\cite{G}) and Neeman's analogous classification result for localising and smashing subcategories of the derived category $\mathsf{D}(\mathsf{Mod}(R))$ (\cite{N}). More relevant for us is the classification of compactly generated t-structures of \cite{AJS}. This is done in terms of certain decreasing sequences of subsets of $\mathsf{Spec}(R)$. Recently, it was shown in \cite{HN} that these same sequences classify pure-injective cosilting objects, opening the door to a concrete approach to mutation in this setting. 

This paper is structured as follows: we start in Section \ref{prelim} by setting up some notation and reviewing some facts about t-structures and torsion pairs; in Section \ref{der cat} we survey the classification of (intermediate) compactly generated t-structures for a commutative noetherian ring and its relation with cosilting theory; in Section \ref{mutation section} we review the main principles of the theory of mutation for cosilting objects, based on \cite{ALSV}. In Section \ref{der eq section} we review some ideas concerning derived equivalences and survey results from \cite{PaV} concerning t-structures inducing derived equivalences in the derived category of a commutative noetherian ring. Finally, in Section \ref{comm mutation} we use a classification result made available in Section \ref{der eq section} to discuss explicit some computations of cosilting mutation in the derived category of a commutative noetherian ring. 

The material surveyed in this article is essentially contained in \cite{PaV} and \cite{ALSV}. We refer to the article of Rosanna Laking in this same volume (\cite{L-Proc}) concerning other aspects of this same concept of mutation in the context of finite-dimensional algebras.

\section{Preliminaries}\label{prelim}

\subsection{Notation}
All subcategories considered are strict and full. In an additive category $\mathcal{A}$ with arbitrary (set-indexed) products, given an object $X$ we denote by $\mathsf{Prod}(X)$ the subcategory formed by all summands of products of $X$. In the following let $\mathcal{X}$ and $\mathcal{Y}$ be subcategories of $\mathcal{A}$. We denote by $\mathcal{X}^\perp$ the subcategory of $\mathcal{A}$ formed by the objects $Y$ for which $\mathsf{Hom}_\mathcal{A}(X,Y)=0$ for all $X$ in $\mathcal{X}$. The notation ${}^\perp\mathcal{X}$ stands for the dual definition. 

If $\mathcal{A}$ is triangulated and $I$ is a subset of $\mathbb{Z}$, then we define 
\begin{equation}\nonumber
\begin{split}
\mathcal{X}^{\perp_I}:=&\{Y\in\mathcal{A}\colon \mathsf{Hom}(X,Y[i])=0,\ \forall i\in I\}\\
{}^{\perp_I}\mathcal{X}:=&\{Y\in\mathcal{A}\colon \mathsf{Hom}(Y,X[i])=0,\ \forall i\in I\}.
\end{split}
\end{equation}
We often replace $I$ by the symbols $\geq 0$, $\leq 0$, $>0$, $<0$, $\neq 0$, with the obvious meaning. We write $\mathcal{X}\ast \mathcal{Y}$ for the subcategory formed by all objects $A$ of $\mathcal{A}$ for which there are $X$ in $\mathcal{X}$, $Y$ in $\mathcal{Y}$ and a triangle
\begin{equation}\nonumber
\xymatrix{X\ar[r]&A\ar[r]&Y\ar[r]&X[1].}
\end{equation}

If $\mathcal{A}$ is abelian, we denote its derived category by $\mathsf{D}(\mathcal{A})$ and its bounded derived category by $\mathsf{D}^b(\mathcal{A})$. If $\mathcal{A}$ is cocomplete, we consider the subcategory of \textbf{finitely presented objects} $\mathsf{fp}(\mathcal{A})$ formed by those $X$ for which $\mathsf{Hom}_\mathcal{A}(X,-)$ commutes with direct limits. Particular focus will be given to \textbf{Grothendieck abelian categories}, i.e., cocomplete abelian categories with exact direct limits and a generator. Recall that a Grothendieck category $\mathcal{A}$ is said to be \textbf{locally coherent} if there is a generating set of finitely presented objects and $\mathsf{fp}(\mathcal{A})$ is abelian. For a ring $R$, we denote by $\mathsf{Mod}(R)$ the category of right $R$-modules and, if $R$ is coherent, we denote by $\mathsf{mod}(R)$ the subcategory $\mathsf{fp}(\mathsf{Mod}(R))$. We will use the shorter version $\mathsf{D}(R)$ for the derived category $\mathsf{D}(\mathsf{Mod}(R))$, and $\mathsf{D}^b(R)$ for the bounded derived category $\mathsf{D}^b(\mathsf{mod}(R))$.

\subsection{Comments on the setup}
We will mostly be working on the derived category of a commutative noetherian ring $R$. Nevertheless, some of our results (mostly introductory ones) hold in a much larger generality and when this generality comes at very little cost to the reader, we shall state the result in such general setting. At points, nevertheless, we may choose to state a result in a more restricted setting, should that contribute to ease the reading of some technical points. We will refer to the literature when appropriate. 

A general framework that we commonly use is that of compactly generated triangulated categories. Recall that a triangulated category $\mathcal{D}$ is \textbf{compactly generated} if if admits set-indexed coproducts and the full subcategory of compact objects, denoted by $\mathcal{D}^c$, is a generating subcategory, i.e., if $(\mathcal{D}^c)^\perp=0$. In a compactly generated triangulated category we have a natural notion of pure-injective objects: an object $X$ is said to be \textbf{pure-injective} if the functor $\mathbf{y}X:=\mathsf{Hom}_\mathcal{D}(-,X)_{|\mathcal{D}^c}$ is an injective functor in the locally coherent Grothendieck category $\mathsf{Mod}(\mathcal{D}^c)$ of contravariant additive functors from $\mathcal{D}^c$ to the category of abelian groups. For more details on this functor category, pure-injective objects and the theory of purity in compactly generated triangulated categories in general, we refer to \cite{Herzog} and \cite{Krause-TC}. Note that the derived category of the category of left/right modules over any ring is a compactly generated triangulated category, whose compact objects are precisely those isomorphic to bounded complexes of finitely generated projective modules.

\subsection{A glossary of torsion pairs and t-structures}\label{glossary}

This paper is focused on the interaction of t-structures in the derived category of a commutative noetherian ring with torsion pairs in the hearts of such t-structures. For the definitions of these three concepts (t-structure; heart; torsion pair) and basic properties (the heart is abelian; a t-structure induces a cohomological functor to the heart)  we refer to \cite{BBDG}. In this subsection we setup a brief list of qualifiers (tagged with some initials in square brackets) that we apply to either t-structures, their hearts or torsion pairs in their hearts throughout the paper. 

\subsubsection{t-structures and cosilting}
Let $\mathcal{D}$ be a compactly generated triangulated category with $\mathcal{D}^c$ its full subcategory of compact objects. 
\begin{itemize}
\item A t-structure $\mathbb{T}:=(\mathcal{U},\mathcal{V})$ in $\mathcal{D}$ is \textbf{compactly generated} if $(\mathcal{U}\cap \mathcal{D}^c)^\perp=\mathcal{V}$.
\item A t-structure $\mathbb{T}:=(\mathcal{U},\mathcal{V})$ in $\mathcal{D}$ is \textbf{nondegenerate} if $\cap_{n\in\mathbb{Z}}\mathcal{U}[n]=0=\cap_{n\in\mathbb{Z}}\mathcal{V}[n]$.
\item If $\mathcal{D}=\mathsf{D}(R)$ for a ring $R$ or if $\mathcal{D}=\mathsf{D}^b(R)$ for a coherent ring $R$, we say that $\mathbb{T}$ is \textbf{intermediate} if there is $a<b$ such that $\mathsf{D}^{\leq a}\subseteq \mathcal{U}\subseteq \mathsf{D}^{\leq b}$, where $\mathbb{D}:=(\mathsf{D}^{< 0},\mathsf{D}^{\geq 0})$ is the standard t-structure in $\mathsf{D}(R)$ (respectively, in $\mathsf{D}^b(R)$).
\end{itemize}

Two particular sources of t-structures are cosilting/cotilting objects. 
\begin{itemize}
\item An object $C$ of $\mathcal{D}$ is said to be \textbf{cosilting} if $\mathbb{T}_C:=({}^{\perp_{\leq 0}}C, {}^{\perp_{>0}}C)$ is a t-structure. Such a t-structure is then said to be a \textbf{cosilting t-structure}. The assignment sending a cosilting object $C$ to its t-structure $\mathbb{T}_C$ is denoted by $\theta$.
\item An object $C$ of $\mathcal{D}$ is said to be \textbf{cotilting} if it is cosilting and $\mathsf{Prod}(C)$ is contained in the heart of $\theta(C)$. Such a t-structure is then said to be a \textbf{cotilting t-structure}.
\end{itemize}

Two cosilting objects $C_1$ and $C_2$ are said to be \textbf{equivalent} if $\mathsf{Prod}(C_1)=\mathsf{Prod}(C_2)$ or, equivalently, if they yield the same t-structure (\cite{PV, NSZ}). This means that the assignment $\theta$ described above can be defined on equivalence classes of cosilting objects.

Of particular interest to us are pure-injective cosilting objects. In fact, it remains an open question whether every cosilting is pure-injective. If $\mathcal{D}=\mathsf{D}(R)$, it is shown in \cite[Proposition 3.10]{MV} (based on \cite{B} and \cite{S}) that any cosilting object isomorphic to a bounded complex of injective $R$-modules (such cosilting complexes will be called \textbf{bounded}) is automatically pure-injective. These will, in fact, be the main cosilting objects we will be interested in and so, the pure-injectivity condition can be safely assumed and, for simplicity, we will avoid (unless strictly necessary) referring to it. 

Let us fix some notation regarding these definitions, where $\mathcal{D}$ denotes a compactly generated triangulated category and $R$ denotes a ring. 
\begin{itemize}
\item We denote by $\mathsf{Cosilt}^*(\mathcal{D})$ (respectively, $\mathsf{Cotilt}^*(\mathcal{D})$) the set of equivalence classes of pure-injective cosilting (respectively, cotilting) objects. It is not obvious that these are indeed sets: this follows from a bijection between equivalence classes of cosilting objects and certain ideals of morphisms of $\mathcal{D}^c$ (\cite[Theorem 8.16 and Proposition 9.6]{SS20}).
\item  We denote the class of t-structures in $\mathcal{D}$ by $\mathsf{t\mbox{-}str}_\mathcal{D}$. We consider the following subsets of this class: 
\begin{itemize}
\item $\mathsf{t\mbox{-}str}_\mathcal{D}[\mathsf{CG}]$, whose elements are the compactly generated t-structures in $\mathcal{D}$
\item $\mathsf{t\mbox{-}str}_\mathcal{D}[\mathsf{Cosilt^*}]$, whose elements are the t-structures $\theta(C)$ for $C$ in $\mathsf{Cosilt}^*(\mathcal{D})$. The restriction of $\theta$ to $\mathsf{Cosilt}^*(\mathcal{D})$ induces, tautologically, a bijection 
\begin{equation}\nonumber
\theta\colon \mathsf{Cosilt}^*(\mathcal{D})\longrightarrow \mathsf{t\mbox{-}str}_\mathcal{D}[\mathsf{Cosilt^*}].
\end{equation}
\end{itemize}
\item If $\mathcal{D}=\mathsf{D}(R)$, we consider the following subclasses of the classes defined above:
\begin{itemize}
\item $\mathsf{Cosilt}^b(\mathsf{D}(R))$: the subset of bounded cosilting objects;
\item $\mathsf{Cotilt}^b(\mathsf{D}(R))$: the subset of bounded cotilting objects;
\item $\mathsf{t\mbox{-}str}_{\mathsf{D}(R)}[\mathsf{Int}]$: the subclass of intermediate t-structures;
\item $\mathsf{t\mbox{-}str}_{\mathsf{D}(R)}[\mathsf{CG,Int}]:=\mathsf{t\mbox{-}str}_{\mathsf{D}(R)}[\mathsf{Int}]\cap \mathsf{t\mbox{-}str}_{\mathsf{D}(R)}[\mathsf{CG}]$;
\item $\mathsf{t\mbox{-}str}_{\mathsf{D}(R)}[\mathsf{Cosilt^*,Int}]:=\mathsf{t\mbox{-}str}_{\mathsf{D}(R)}[\mathsf{Int}]\cap \mathsf{t\mbox{-}str}_{\mathsf{D}(R)}[\mathsf{Cosilt^*}]$.
\end{itemize}
\end{itemize}

The following theorem summarises some useful properties of compactly generated t-structures, cositling t-structures and their relation.
\begin{theorem}[\cite{AMV3,L,MV,SS20}]\label{t-structures prelim}
Let $\mathcal{D}$ be a compactly generated triangulated category and let $\mathbb{T}=(\mathcal{U},\mathcal{V})$ be a t-structure in $\mathcal{D}$. The following statements hold.
\begin{enumerate}
\item The t-structure $\mathbb{T}$ lies in $\mathsf{t\mbox{-}str}_\mathcal{D}[\mathsf{Cosilt^*}]$ if and only if $\mathbb{T}$ is nondegenerate, $\mathcal{V}$ is closed under coproducts and the heart of $\mathbb{T}$ is a Grothendieck category. 
\item If $\mathbb{T}$ is nondegenerate and compactly generated t-structure in $\mathcal{D}$, then there is a (unique, up to equivalence) pure-injective cosilting object $C$ such that $\mathsf{Prod}(C)=\mathcal{V}^{\perp}[-1]\cap \mathcal{V}$. We denote the injective function assigning $\mathbb{T}$ to $C$ by $\alpha$; it has $\theta$ as a left inverse, i.e. $\mathbb{T}=\theta(\alpha(\mathbb{T}))=({}^{\perp_{\leq 0}}\alpha(\mathbb{T}),{}^{\perp_{> 0}}\alpha(\mathbb{T}))$.
\item if $\mathcal{D}=\mathsf{D}(R)$, then the assignment $\alpha$ defined above restricts to an injective function (also denoted by $\alpha$)
\begin{equation}\nonumber
\alpha\colon \mathsf{t\mbox{-}str}_\mathcal{D}[\mathsf{CG, Int}]\longrightarrow \mathsf{Cosilt}^b(\mathcal{D}) (\subseteq \mathsf{Cosilt}^*(\mathcal{D})).
\end{equation}
\end{enumerate}
\end{theorem}
\begin{proof}
Statement (1) is shown in \cite[Corollary 3.8]{AMV3}, using the Brown Representability Theorem for compactly generated triangulated categories. Statement (2) was first proved in \cite{AMV3} when $\mathcal{D}$ is, in addition, algebraic; 
it was later extended in \cite{L} for when $\mathcal{D}$ lies at the base of stable derivator, and it was finally proved for all compactly generated triangulated categories in \cite[Thereom 8.31]{SS20}. We also refer to \cite{PV} and \cite{NSZ} for details on how to obtain the cosilting object out of a cosilting t-structure. Finally step (3) is an easy observation that follows, for example, from \cite[Theorem 3.14]{MV}. The fact that $\mathsf{Cosilt}^b(\mathcal{D}) \subseteq \mathsf{Cosilt}^*(\mathcal{D})$ was already mentioned - see \cite[Proposition 3.10]{MV}.
\end{proof}

\subsubsection{Hearts}
We denote by $\mathcal{H}$ the assignment that sends a t-structure $\mathbb{T}=(\mathcal{U},\mathcal{V})$ in a triangulated category $\mathcal{D}$ to its heart, which we write as $\mathcal{H}_\mathbb{T}$ and it denotes the subcategory $\mathcal{U}[-1]\cap\mathcal{V}$ of $\mathcal{D}$, with inclusion functor denoted by $\epsilon_\mathbb{T}\colon \mathcal{H}_\mathbb{T}\longrightarrow \mathcal{D}$. We denote the class of hearts in $\mathcal{D}$ by $\mathsf{Heart}_\mathcal{D}$. 
By definition, the assignment $\mathcal{H}\colon \mathsf{t\mbox{-}str}_{\mathcal{D}}\longrightarrow \mathsf{Heart}_\mathcal{D}$ is surjective, but it is well known that a heart $\mathcal{H}_\mathbb{T}$ does not uniquely determine $\mathbb{T}$. For example, any stable t-structure (i.e. a t-structure for which both classes are triangulated subcategories) has a zero heart, and yet there are multiple such t-structures in any non-trivial $\mathcal {D}$ (for example, the pairs $(0,\mathcal{D})$ and $(\mathcal{D},0)$). In other words, the assignment $\mathcal{H}\colon \mathsf{t\mbox{-}str}_\mathcal{D} \longrightarrow \mathsf{Heart}_\mathcal{D}$ is surjective by definition, but usually highly non injective. Still, we have the following easy lemma (analogous to the reconstruction of a bounded t-structure from its heart, see for example \cite[Lemma 2.3]{Br}).

\begin{lemma}\label{int t-str}
Let $R$ be a ring and $\mathbb{T}$ an intermediate t-structure in $\mathsf{D}(R)$. Then the subcategory $\mathcal{H}_\mathbb{T}$ of $\mathcal{D}$ determines the intermediate t-structure $\mathbb{T}$.
\end{lemma}
\begin{proof}
Let $\mathbb{T}=(\mathcal{U},\mathcal{V})$. Without loss of generality, suppose that $\mathsf{D}^{\leq -n}\subseteq \mathcal{U}\subseteq \mathsf{D}^{\leq 0}$ for some $n\geq 0$. Let $X$ be an object in $\mathsf{D}^{\geq -n+1}\cap \mathcal{U}$ (and, therefore, in $\mathcal{V}[n]\cap \mathcal{U}$). By iterated truncations (with respect to $\mathbb{T}$), we obtain that $X$ is an iterated finite extension of shifts of objects in $\mathcal{H}_\mathbb{T}$. Indeed, notice that this iteration stops at the truncation with respect to $\mathcal{U}[n]\subseteq \mathsf{D}^{\leq -n}$, since objects in this subcategory have no maps to $X$. It then follows that every object in $\mathcal{U}$ is an extension of an object in $\mathsf{D}^{\leq -n}$ with an object built from a finite sequence of extensions of objects in shifts of the heart $\mathcal{H}_\mathbb{T}$. Finally, note that $\mathbb{T}$ is completely determined by $\mathcal{U}$.
\end{proof}

We denote by $\mathsf{Heart}_{\mathsf{D}(R)}[\mathsf{CG, Int}]$ the image of $\mathsf{t\mbox{-}str}_{\mathsf{D}(R)}[\mathsf{CG, Int}]$ under $\mathcal{H}$.

\subsubsection{Torsion pairs} A torsion pair $\mathbf{t}:=(\mathcal{T},\mathcal{F})$ in an abelian category $\mathcal{A}$ is said to be 
\begin{itemize}
\item \textbf{hereditary} if $\mathcal{T}$ is closed under subobjects;
\item \textbf{of finite type} if $\mathcal{A}$ is cocomplete and $\mathcal{F}$ is closed under direct limits. 
\end{itemize}
We denote the class of torsion pairs in an abelian category $\mathcal{A}$ by $\mathsf{Tors}_\mathcal{A}$ and we write
\begin{itemize}
\item $\mathsf{Tors}_\mathcal{A}[\mathsf{Her}]$ for the subclass of hereditary torsion pairs;
\item $\mathsf{Tors}_\mathcal{A}[\mathsf{FT}]$ for the subclass of torsion pairs of finite type (if $\mathcal{A}$ is cocomplete);
\item $\mathsf{Tors}_\mathcal{A}[\mathsf{Her, FT}]:=\mathsf{Tors}_\mathcal{A}[\mathsf{Her}]\cap \mathsf{Tors}_\mathcal{A}[\mathsf{FT}]$  (if $\mathcal{A}$ is cocomplete).
\end{itemize} 

If $\mathcal{A}$ is a locally coherent Grothendieck category, we may wonder whether the classes $\mathsf{Tors}_\mathcal{A}$ and $\mathsf{Tors}_{\mathsf{fp}(\mathcal{A})}$ are related in any sensible way. 

\begin{theorem}\cite[Lemma 4.4]{CBlfp}\label{torsion pairs prelim}
Let $\mathcal{A}$ denote a locally finitely presented Grothendieck category. There is an injective assignment
\begin{align*}
\mathsf{Lift}\colon \mathsf{Tors}_{\mathsf{fp}(\mathcal{A})}\longrightarrow& \mathsf{Tors}_{\mathcal{A}}[FT]\\ (\mathsf{t},\mathsf{f})\mapsto&({}^\perp(\mathsf{t}^\perp),(\mathsf{t}^\perp))=(\varinjlim \mathsf{t},\varinjlim \mathsf{f})
\end{align*}
where $\varinjlim \mathsf{t}$ and $\varinjlim \mathsf{f}$ denote the closure under direct limits in $\mathcal{A}$ of the subclasses $\mathsf{t}$ and $\mathsf{f}$ of $\mathsf{fp}(\mathcal{A})$, respectively. Furthermore, if $\mathcal{A}$ is a locally noetherian Grothendieck category (i.e. if all objects in $\mathsf{fp}(\mathcal{A})$ are noetherian), then $\mathsf{Lift}$ is a bijection. 
\end{theorem}

\section{A guide to compactly generated t-structures for a commutative noetherian ring}\label{der cat}

For a commutative noetherian ring $R$, consider its prime spectrum, which 
we denote by $\mathsf{Spec}(R)$. This is a poset ordered by inclusion, and it is endowed with two natural topologies, naturally dual to each other. Indeed, for a prime ideal $\mathfrak{p}$, let $V(\mathfrak{p})$ denote the set of prime ideals of $R$ that contain $\mathfrak{p}$. Then
\begin{itemize}
\item The Zariski topology is the unique topology for which the closed sets are exactly the finite unions of sets of the form $V(\mathfrak{p})$;
\item The Hochster dual topology is the unique topology for which the sets of the form $V(\mathfrak{p})$ are a basis of open subsets. 
\end{itemize}
Note that the open sets of the Hochster dual topology are therefore the upper sets of the poset $\mathsf{Spec}(R)$, i.e. the subsets $V$ of $\mathsf{Spec}(R)$ for which if $\mathfrak{p}$ lies in $V$ and $\mathfrak{p}\subseteq \mathfrak{q}$, then also $\mathfrak{q}$ lies in $V$. We will denote the set of these subsets by $\mathscr{O}^H(\mathsf{Spec}(R))$, and they are called \textbf{specialisation-closed}. These subsets of $\mathsf{Spec}(R)$ play an important role in classifying hereditary torsion pairs in $\mathsf{Mod}(R)$ (see Example \ref{standard}). 

When studying t-structures in $\mathsf{D}(R)$, the following slightly more intricate combinatorial objects are useful.

\begin{lemma}\label{combinatorics}\cite[Proposition 4.1]{Sta}\cite[Proposition 4.3]{Tak}\cite[Remark 2.10]{HNS} Let $R$ be a commutative noetherian ring. There is a bijection between 
\begin{enumerate}
\item Decreasing functions $\varphi\colon \mathbb{Z}\longrightarrow 2^{\mathsf{Spec}(R)}$ such that $\varphi(n)$ lies in $\mathscr{O}^H(\mathsf{Spec}(R))$ for all $n$ in $\mathbb{Z}$ (these are known as \textbf{sp-filtrations} of $\mathsf{Spec}(R)$).
\item Increasing functions $f\colon \mathsf{Spec}(R)\longrightarrow \overline{\mathbb{Z}}:=\mathbb{Z}\cup\{\pm \infty\}$.
\end{enumerate}
Moreover, this bijection restricts to a bijection between sp-filtrations $\varphi$ of $\mathsf{Spec}(R)$ for which there are $a<b$ integers such that $\varphi(a)=\mathsf{Spec}(R)$ and $\varphi(b)=\emptyset$ (we will call them \textbf{bounded}) and \textbf{bounded} increasing functions $f\colon \mathsf{Spec}(R)\longrightarrow \mathbb{Z}$, i.e. those whose image is contained in an interval $[a,b]$ ($a<b$ integers). 
\end{lemma}

The set formed by the bounded sp-filtrations will be denoted by $\mathsf{sp\mbox{-}filt}^b(R)$; the bounded increasing functions mentioned above are bounded poset homomorphisms from $\mathsf{Spec}(R)$ to $\mathbb{Z}$, and we denote the set of those by $\mathsf{Hom}^b_\mathsf{pos}(\mathsf{Spec}(R),\mathbb{Z})$. Below, we define the inverse assignments giving the bijection set out in the theorem:

\begin{equation}\nonumber
\begin{split}
\mathsf{f}\colon \mathsf{sp\mbox{-}filt}^b(R)\longrightarrow \mathsf{Hom}^b_\mathsf{pos}(\mathsf{Spec}(R),\mathbb{Z}), &\ \ \  \varphi\mapsto \mathsf{f}_\varphi\\
\Phi\colon  \mathsf{Hom}^b_\mathsf{pos}(\mathsf{Spec}(R),\mathbb{Z})\longrightarrow \mathsf{sp\mbox{-}filt}^b(R), &\ \ \  f\mapsto\Phi_f.
\end{split}
\end{equation}

Given a bounded sp-filtration $\varphi$, note that $\mathsf{Spec}(R)$ is the $\mathbb{Z}$-indexed disjoint union of the subsets $\varphi(n-1)\setminus \varphi(n)$. Define a function $\mathsf{f}_\varphi\colon \mathsf{Spec}(R)\longrightarrow \mathbb{Z}$ by assigning to each prime ideal $\mathfrak{p}$ the (unique) integer $\mathsf{f}_\varphi(\mathfrak{p}):=n$ for which $\mathfrak{p}$ lies in $\varphi(n-1)\setminus \varphi(n)$. It can then be checked that $\mathsf{f}_\varphi$ is an increasing function (essentially due to the fact that $\phi(n)$ is specialisation-closed for all $n$ in $\mathbb{Z}$). 
Note that if $\varphi(a)=\mathsf{Spec}(R)$ and $\varphi(b)=\emptyset$, then the image of $\mathsf{f}_\varphi$ is contained in the integer interval $[a,b[$. 
For the inverse assignment, given an increasing function $f\colon \mathsf{Spec}(R)\longrightarrow \mathbb{Z}$, define a filtration $\Phi_f$ of $\mathsf{Spec}(R)$ as follows: $\Phi_f(n):=f^{-1}(]n,+\infty[)$. It can then be easily checked that this is an sp-filtration and, moreover, that since $f$ is assumed to have a bounded image, say contained in $]a,b]$ for some integers $a<b$, then $\Phi_f(a)=f^{-1}(]a,+\infty[)=\mathsf{Spec}(R)$ and $\Phi_f=f^{-1}(]b,+\infty[)=\emptyset$. This shows that $\Phi_f$ is, indeed, a bounded sp-filtration.

\begin{example}
Note that if $R$ has finite Krull dimension $n$, the \textbf{height function} sending each prime in $\mathsf{Spec}(R)$ to its height in a bounded poset homomophism, whose image lies precisely in the interval $[0,n]$.
\end{example}

The above combinatorial data turns out to classify compactly generated t-structures in the derived category of a commutative noetherian ring $R$. Denoting by $k(\mathfrak{p})$ the residue field of $R$ at $\mathfrak{p}$, we define the \textbf{support} of an object $X$ of $\mathsf{D}(R)$ as
$$\mathsf{supp}(X):=\{\mathfrak{p}\in\mathsf{Spec}(R)\colon k(\mathfrak{p})\otimes_R^\mathbb{L}X\neq 0\}$$
For any subcategory $\mathcal{X}$ of $\mathsf{D}(R)$, we write $\mathsf{supp}(\mathcal{X})$ to denote the union, over all $X$ in $\mathcal{X}$ of $\mathsf{supp}(X)$. The following theorem combines a series of results from the literature.

\begin{theorem}\label{diagram}
Let $R$ be a commutative noetherian ring. The diagram
\begin{equation}\nonumber
\xymatrix{\mathsf{t\mbox{-}str}_{\mathsf{D}^b(R)}[\mathsf{Int}]\ar[rrrr]^{\alpha\ \circ \ \mathsf{t\mbox{-}Lift}}\ar[dd]^{\mathsf{t\mbox{-}Lift}}&&&&\mathsf{Cosilt}^b(\mathsf{D}(R))\ar@/^1.3pc/[ddllll]^{\theta}\ar[dd]_{\mathcal{H}\circ \theta}\\ \\ \mathsf{t\mbox{-}str}_{\mathsf{D}(R)}[\mathsf{CG, Int}]\ar@/^1.3pc/[uurrrr]^{\alpha}\ar[rrrr]_{\mathcal{H}}\ar[d]^{\beta}&&&& \mathsf{Heart}_{\mathsf{D}(R)}[\mathsf{CG,Int}]\ar[d]_{\gamma}\\ \mathsf{sp\mbox{-}filt}^b(R)\ar[rrrr]^{\mathsf{f}}&&&&\mathsf{Hom}^b_\mathsf{pos}(\mathsf{Spec}(R),\mathbb{Z})}
\end{equation}
commutes, $\mathsf{t\mbox{-}Lift}$ is an injection and  $\alpha$, $\beta$, $\gamma$, $\theta$, $\mathcal{H}$ and $\mathsf{f}$ are bijections defined by:
 \begin{itemize} 
  \item $\alpha$, $\theta$, $\mathcal{H}$ and $\mathsf{f}$ are the assignments already defined: 
  \begin{itemize}
  \item $\alpha$ sends a compactly generated intermediate t-structure $\mathbb{T}$ to the (unique up to equivalence) bounded cosilting complex $\alpha(\mathbb{T})$ such that $\mathbb{T}=({}^{\perp_{\leq 0}}\alpha(\mathbb{T}),{}^{\perp_{>0}}\alpha(\mathbb{T}))$; 
  \item $\theta$ sends a cosilting complex $C$ to the associated cosilting t-structure, and it is the inverse of $\alpha$;
  \item $\mathcal{H}$ sends a t-structure $\mathbb{T}$ to its heart $\mathcal{H}_\mathbb{T}$;
  \item $\mathsf{f}$ sends a bounded sp-filtration $\varphi$ to a bounded morphism of posets $\mathsf{f}_\varphi$;
  \end{itemize}
 \item $\mathsf{t\mbox{-}Lift}(u,v):=({}^\perp(u^{\perp}),u^{\perp})=(\mathsf{hocolim}(u),\mathsf{hocolim}(v))$;
 \item $\beta(\mathcal{U},\mathcal{V}):=(\varphi\colon \mathbb{Z}\longrightarrow 2^{\mathsf{Spec}(\mathbb{R})}; n\mapsto \mathsf{supp}(H^n(\mathcal{U})))$;
 \item $\gamma(\mathcal{H}_\mathbb{T}):=(\psi\colon \mathsf{Spec}(R)\longrightarrow \mathbb{Z}; \mathfrak{p}\mapsto \psi(\mathfrak{p}))$, where $\psi(\mathfrak{p})$  is the unique integer for which $k(\mathfrak{p})[-\psi(\mathfrak{p})]$ lies in $\mathcal{H}_\mathbb{T}$.
 \end{itemize}
\end{theorem}
\begin{proof}
We articulate the results in the literature that yield the diagram. Note that $\mathcal{H}$ is a bijection from Lemma \ref{int t-str}.
\begin{enumerate}
\item $\mathsf{t\mbox{-}Lift}$ is well-defined (i.e., the image is indeed a compactly generated and intermediate t-structure) by \cite{MZ}. It is an injection with a left inverse given by the intersection with $\mathsf{D}^b(R)$. Note that $\mathsf{t\mbox{-}Lift}$ is a t-structure version of the assignment $\mathsf{Lift}$ defined for torsion pairs in Theorem \ref{torsion pairs prelim}.
\item The map $\alpha$ is well-defined and injective by Theorem \ref{t-structures prelim}(3) and its surjectivity is a version of the so-called \textit{telescope problem for t-structures}. Starting with a pure-injective cosilting object $C$, it is well-known that the associated t-structure $\theta(C)$ is homotopically smashing (see \cite{SSV} for the relevant concept). There is a priori no reason for it to be compactly generated, but it is shown in \cite[Theorem 1.1]{HN} that for a commutative noetherian ring this is the case. The intermediate condition follows from the boundedness of the cosilting complex. 
\item The fact that $\beta$ is a bijection follows from \cite[Theorem 3.10]{AJS}. The inverse assignment is the obvious one: for each sp-filtration $\varphi$ consider the class 
\begin{equation}\nonumber
\mathcal{U}_\varphi:=\{X\in \mathsf{D}(R)\colon \mathsf{supp}(H^n(X))\in \varphi(n)\}.
\end{equation}
It is shown in \cite{AJS} that $(\mathcal{U}_\varphi,\mathcal{U}_\varphi^\perp)$ is indeed a compactly generated t-structure in $\mathsf{D}(R)$. Once again, if we restrict ourselves to intermediate t-structures it is easy to see that we get under $\beta$ a bijection with bounded sp-filtrations. 
\item The fact that $\gamma$ is well-defined follows from \cite[Proposition 4.7]{PaV}, based on results of \cite{HN}. It is indeed the case that for the heart $\mathcal{H}_\mathbb{T}$ of a compactly generated intermediate t-structure $\mathbb{T}$ in $\mathsf{D}(R)$, for each residue field of $R$ there is one and one only shift of it that lies in $\mathcal{H}_\mathbb{T}$. Moreover, as shown in \cite{PaV}, that shift is precisely determined by the symmetric value of $\mathsf{f}_\varphi$, where $\varphi$ is the sp-filtration associated to $\varphi$ under the correspondence of Lemma \ref{combinatorics}. Note that this is the essence of the commutativity of the bottom square of the diagram. Since all the other maps in the square are bijections, then so is $\gamma$. \qedhere
\end{enumerate}
\end{proof}

\begin{remark}\label{loc coherent heart}
It is also shown in \cite{MZ} that if $\mathbb{S}$ is an intermediate t-structure in $\mathsf{D}^b(R)$ for a coherent ring $R$ and $\mathbb{T}=\mathsf{t\mbox{-}Lift}(\mathbb{S})$, then the heart $\mathcal{H}_\mathbb{T}$ is a locally coherent Grothendieck category, whose subcategory of finitely presented objects $\mathsf{fp}(\mathcal{H}_\mathbb{T})=\mathcal{H}_\mathbb{S}$. 
\end{remark}
 
We shall use the diagram in the theorem above as a guide for the next sections. We will define an operation of mutation in $\mathsf{Cosilt}^b(\mathsf{D}(R))$, and attempt to translate it to combinatorial data in $\mathsf{Spec}(R)$ following the diagram. 

\begin{example}
Consider $\mathbb{D}$ to be the standard t-structure in $\mathsf{D}^b(R)$. Then we have that 
\begin{itemize}
\item $\mathsf{t\mbox{-}Lift}(\mathbb{D})$ is the standard t-structure in $\mathsf{D}(R)$;
\item $\mathcal{H} \circ \mathsf{t\mbox{-}Lift}(\mathbb{D})$ coincides with the complexes $X$ in $\mathsf{D}(R)$ for which $H^k(X)=0$ for all $k\neq 0$, and it is equivalent to $\mathsf{Mod}(R)$;
\item $\alpha \circ \mathsf{t\mbox{-}Lift}(\mathbb{D})$ is the (equivalence class of the) injective cogenerator of $\mathsf{Mod}(R)$;
\item $\beta\circ  \mathsf{t\mbox{-}Lift}(\mathbb{D})$ is the sp-filtration defined by 
\begin{equation}\nonumber
\varphi(n)=\begin{cases} \mathsf{Spec}(R)& {\rm if\ } n\leq -1\\ \emptyset & {\rm otherwise;}\end{cases}
\end{equation}
\item $\gamma\circ \mathcal{H}\circ  \mathsf{t\mbox{-}Lift}(\mathbb{D})$ is the bounded homomorphism of posets $\psi\colon \mathsf{Spec}(R)\longrightarrow \mathbb{Z}$ which is constant equal to $0$. 
\end{itemize}
\end{example}

\section{Mutation}\label{mutation section}

In this paper we discuss right mutation only. Left mutation is an inverse operation to right mutation, and it is also discussed in detail in \cite{ALSV}. As mentioned in the introduction, mutation is a process by which we change a component of an object of a certain nature to produce an object of the same nature. We introduce mutation of cosilting objects from two points of view and claim in Theorem \ref{pathways to mutation} that they are equivalent. 

\subsection{Mutation via approximation theory}
We begin with a traditional point of view on mutation: for a fixed object pick a part of it that we want to remain unchanged, (right or left) approximate the object with respect to the part, and consider the triangle associated to such a map. The direct sum of the two new objects appearing in that triangle is then called the left/right mutation of the fixed object with respect to the part. This is essentially the recipe followed for mutating exceptional collections in algebraic geometry (see, for example, \cite{Bondal}), or silting/cluster-tilting objects in representation theory (see, for example, \cite{BMRRT,AI}). 

We will discuss right mutations and, thus, we will only care about right approximations or precovers. Given a subcategory $\mathcal{X}$ of an additive category $\mathcal{A}$, a map $f\colon X\longrightarrow A$ is said to be an \textbf{$\mathcal{X}$-precover} of $A$ if $X$ lies in $\mathcal{X}$ and $\mathsf{Hom}_\mathcal{A}(Y,f)$ is surjective for all $Y$ in $\mathcal{X}$. Such an $\mathcal{X}$-precover $f$ is said to be an \textbf{$\mathcal{X}$-cover} if $f$ is right minimal, i.e. if for any endomorphism $g$ of $X$, if $f\circ g=f$ then $g$ is an isomorphism. If every object of $\mathcal{A}$ admits an $\mathcal{X}$-(pre)cover, then we say that $\mathcal{X}$ is \textbf{(pre)covering} in $\mathcal{A}$.

We can now describe mutation of pure-injective cosilting objects via approximations. Let $\mathcal{D}$ be a compactly generated triangulated category and $C$ an object in $\mathsf{Cosilt}^*(\mathcal{D})$. Consider the following procedure that, out of an input subject to a mutability condition, produces an output.
\begin{itemize}
\item \textbf{Input:} Consider a subcategory $\mathscr{E}$ of $\mathsf{Prod}(C)$ satisfying $\mathsf{Prod}(\mathscr{E})=\mathscr{E}$.
\item \textbf{Mutability condition:} We say that $C$ admits a right mutation with respect to $\mathscr{E}$ if $C$ admits an $\mathscr{E}$-cover.
\item \textbf{Output:} Let $\Phi\colon E_0\longrightarrow C$ be an $\mathscr{E}$-cover for $C$, and consider the triangle
\begin{equation}\nonumber
\xymatrix{E_1\ar[r]&E_0\ar[r]^\Phi&C\ar[r]&E_1[1].}
\end{equation}
Then we define the \textbf{right mutation of $C$ at $\mathscr{E}$} to be $\mu^{-}_\mathscr{E}(C):=E_0\oplus E_1$.
\end{itemize}
The following theorem justifies why this operation is well-defined within $\mathsf{Cosilt}^*(\mathcal{D})$.

\begin{theorem}\cite[Proposition 4.5]{ALSV}
Given $C$ in $\mathsf{Cosilt}^*(\mathcal{D})$ and $\mathscr{E}\subseteq\mathsf{Prod}(C)$ as above satisfying the mutability condition above, then the output $\mu^{-}_\mathscr{E}(C)$ lies in $\mathsf{Cosilt}^*(\mathcal{D})$.
\end{theorem}
\subsection{Mutation via torsion pairs}
We now shift our focus from objects to categorical decompositions, namely t-structures and torsion pairs. The idea of creating a t-structure out of an old one, keeping part of the heart and changing some other part (these two parts forming a torsion pair, and in a right mutation we will want to \textit{keep} the torsionfree class and \textit{change} the torsion class) follows from \cite{HRS}. This process is called (right) \textbf{HRS-tilt}, and can be succinctly described as follows. 

\begin{theorem}\cite{HRS}\label{HRS}
Let $\mathbb{T}=(\mathcal{U},\mathcal{V})$ be a t-structure in a triangulated category $\mathcal{D}$ and let $\mathbf{t}:=(\mathcal{T},\mathcal{F})$ be a torsion pair in $\mathcal{H}_\mathbb{T}$. Then the pair $(\mathcal{U}\ast  \mathcal{T},\mathcal{F}\ast (\mathcal{V}[-1]))$ is a t-structure with heart $\mathcal{F}\ast(\mathcal{T}[-1])$, called the \textbf{right HRS-tilt of $\mathbb{T}$ at $\mathbf{t}$}.
\end{theorem}

If we want this process to restrict to a given class of t-structures, say cosilting t-structures, then a mutability condition must be imposed. We can indeed describe mutation of cosilting t-structures (coming from pure-injective cosilting objects) via HRS-tilts. Let $\mathcal{D}$ be a compactly generated triangulated category and $\mathbb{T}$ a t-structure in $\mathsf{t\mbox{-}str}_\mathcal{D}[\mathsf{Cosilt^*}]$. Consider the following procedure that, out of an input subject to a mutability condition, produces an output.
\begin{itemize}
\item \textbf{Input:} Consider a torsion pair $\mathbf{t}=(\mathcal{T},\mathcal{F})$ in $\mathsf{Tors}_{\mathcal{H}_\mathbb{T}}[\mathsf{Her}]$.
\item \textbf{Mutability condition:} We say that $\mathbb{T}$ admits a right mutation with respect to $\mathbf{t}$ if $\mathbf{t}$ is of finite type, i.e. if it lies in $\mathsf{Tors}_{\mathcal{H}_\mathbb{T}}[\mathsf{Her, FT}]$. 
\item \textbf{Output:} Consider the the right HRS-tilt of $\mathbb{T}$ at $\mathbf{t}$ (see Theorem \ref{HRS}) given by the pair $\mu^-_\mathbf{t}(\mathbb{T}):=(\mathcal{U}\ast \mathcal{T},\mathcal{F}\ast (\mathcal{V}[-1]))$, which we define to be \textbf{the right mutation of $\mathbb{T}$ at $\mathbf{t}$}.
\end{itemize}
The following theorem justifies why this operation is well-defined within $\mathsf{t\mbox{-}str}_\mathcal{D}[\mathsf{Cosilt}^*]$. This result relies on a generalisation of \cite[Theorem 1.2]{PS}, which discusses when the right HRS-tilt of a Grothendieck heart is still Grothendieck.

\begin{theorem}\cite[Theorem 4.3, Proposition 4.5]{ALSV}
Given the input $(\mathbb{T},\mathbf{t})$ as above satisfying the mutability condition above, then the output $\mu^-_\mathbf{t}(\mathbb{T})$ lies in $\mathsf{t\mbox{-}str}_\mathcal{D}[\mathsf{Cosilt}^*]$.
\end{theorem}

\subsection{The two approaches are equivalent}
As announced in the preamble of this section, we discuss why the approaches to mutation defined above are two points of view on essentially the same transformation. For a fixed pure-injective cosilting complex $C$, denote by $\theta(C)$ the associated cosilting t-structure, with $\mathcal{H}_C:=\mathcal{H}_{\theta(C)}={}^{\perp_{\neq 0}}C$ the associated (Grothendieck) heart and $\mathsf{H}^0_C\colon \mathcal{D}\longrightarrow \mathcal{H}_C$ the associated cohomological functor. 
\begin{theorem}\cite[Theorems 3.5 and 4.9]{ALSV}\label{pathways to mutation}
In a compactly generated triangulated category $\mathcal{D}$, the above pathways to cosilting mutation are equivalent. More precisely, given $C$ in $\mathsf{Cosilt}^*(\mathcal{D})$ and $\mathscr{E}=\mathsf{Prod}(\mathscr{E})$ in $\mathsf{Prod}(C)$, we have that 
\begin{enumerate}
\item $C$ admits a right mutation with respect to $\mathscr{E}$ if and only if $\theta(C)$ admits a right mutation with respect to $\mathbf{t}_\mathscr{E}:=({}^\perp H^0_C(C),\mathsf{Cogen}(H^0_C(C)))$.
\item If $C$ admits a right mutation with respect to $\mathscr{E}$, then $\theta(\mu^-_\mathscr{E}(C))=\mu^{-}_{\mathbf{t}_\mathscr{E}}(\theta(C))$.

\end{enumerate}
\end{theorem}
\begin{proof}
We begin by recalling that the assignment $\theta$ is a bijection. Fix $C$ a pure-injective cosilting object in $\mathcal{D}$ and $\mathbb{T}:=\theta(C)$ its associated t-structure.

\textbf{Equivalence of inputs}: The cohomological functor $\mathsf{H}^0_C$ associated to $\mathbb{T}$ induces an equivalence of categories between $\mathsf{Prod}(C)$ and the subcategory $\mathsf{Inj}(\mathcal{H}_{\mathbb{T}})$ of injective objects in $\mathcal{H}_\mathbb{T}$, which we know to be a Grothendieck category. Given a subcategory $\mathscr{E}=\mathsf{Prod}(\mathscr{E})$ of $\mathsf{Prod}(C)$, $H^0_\mathbb{T}(\mathscr{E})$ is then a class of injective objects in $\mathcal{H}_\mathbb{T}$ that, therefore, cogenerates a hereditary torsion pair $({}^\perp H^0_\mathbb{T}(\mathscr{E}),\mathsf{Cogen}(H^0_\mathbb{T}(\mathscr{E})))$ where $\mathsf{Cogen}(H^0_\mathbb{T}(\mathscr{E}))$ stands for the class of subobjects in $\mathcal{H}_\mathbb{T}$ of products objects in $H^0_\mathbb{T}(\mathscr{E})$. 

\textbf{Equivalence of mutability conditions}:  This is Statement (1) of our theorem; we refer to the cited reference for a complete proof. 

\textbf{Equivalence of output}: It is shown in \cite[Theorem 4.9]{ALSV} that the t-structure associated to the cosilting object $\mu^{-}_\mathscr{E}(C)$ is precisely the right HRS-tilt of $\mathbb{T}$ at the torsion pair $\mathbf{t}_\mathscr{E}$. Moreover, if $\mathscr{E}_1$ and $\mathscr{E}_2$ are subcategories of $\mathsf{Prod}(C)$ satisfying the mutability condition such that $\mathbf{t}_{\mathscr{E}_1}=\mathbf{t}_{\mathscr{E}_2}$, then $\mu^-_{\mathscr{E}_1}(C)$ and $\mu^-_{\mathscr{E}_2}(C)$ are equivalent.
\end{proof}

\begin{remark}
The condition for a subcategory $\mathscr{E}=\mathsf{Prod}(\mathscr{E})$ of $\mathsf{Prod}(C)$ to satisfy the mutability condition can be phrased topologically (in the Ziegler spectrum) if $\mathcal{H}_{\theta(C)}$ is a locally coherent Grothendieck category. We refer to \cite{ALS} for details. 
\end{remark}

\subsection{The commutative noetherian case}
Let $R$ be a commutative noetherian ring and $\mathcal{D}=\mathsf{D}(R)$. We try to understand how the action of mutation translates to actions on the vertices of the diagram of Theorem \ref{diagram}, in particular to those of combinatorial nature: $\mathsf{sp\mbox{-}filt}^b(R)$ and $\mathsf{Hom}^b_{\mathsf{pos}}(\mathsf{Spec}(R),\mathbb{Z})$. To make this translation possible, we first ensure that the input for a mutation is determined by a subset of $\mathsf{Spec}(R)$. The proof of the following lemma relies on Neeman's classification of localising subcategories in $\mathsf{D}(R)$ (see \cite{N} for details).

\begin{lemma}\cite[Proposition 4.1 and Theorem 4.5]{PaV}\label{basics on support for hearts}
Let $R$ be a commutative noetherian ring. Let $C$ be a bounded cosilting object in $\mathsf{D}(R)$ with associated heart $\mathcal{H}_{\theta(C)}$ and associated bounded homomorphism of posets $\psi_C:=\gamma(\mathcal{H}_{\theta(C)})\colon \mathsf{Spec}(R)\longrightarrow \mathbb{Z}$. The following statements hold for the assingment
\begin{equation}\nonumber
\hat{\mathfrak{s}}\colon \mathsf{Tors}_{\mathcal{H}_{\theta(C)}}[Her]\longrightarrow 2^{\mathsf{Spec}(R)}
\end{equation}
sending a hereditary torsion pair $\mathbf{t}=(\mathcal{T},\mathcal{F})$ in $\mathcal{H}_{\theta(C)}$ to $\mathsf{supp}(\mathcal{T})$. 
\begin{enumerate}
\item $\hat{\mathfrak{s}}$ is injective, with left inverse given by the formula
\begin{equation}\nonumber
\mathsf{supp}^{-1}(\hat{\mathfrak{s}}(\mathbf{t}))\cap \mathcal{H}_\mathbb{T}=\mathcal{T}.
\end{equation} 
\item Given $\mathbf{t}=(\mathcal{T},\mathcal{F})$ in $\mathsf{Tors}_{\mathcal{H}_{\theta(C)}}[\mathsf{Her}]$, we have
\begin{equation}\nonumber
\hat{\mathfrak{s}}(\mathbf{t})=\mathsf{supp}(\mathcal{T}):=\{\mathfrak{p}\in\mathsf{Spec}(R)\colon k(\mathfrak{p})[-\psi_C(\mathfrak{p})]\in\mathcal{T}\}.
\end{equation}
\item $\mathscr{O}^H(\mathsf{Spec}(R))$ is contained in the image of $\hat{\mathfrak{s}}$, and for $V$ in $\mathscr{O}^H(\mathsf{Spec}(R))$ we have that $\mathsf{supp}^{-1}(V)\cap\mathcal{H}_{\theta(C)}$ is a hereditary torsion class in $\mathcal{H}_{\theta(C)}$.
\end{enumerate}
\end{lemma}

Note that the Lemma above does not tell us precisely what the image of $\mathfrak{s}$ is. Nevertheless, the first statement helpfully asserts that a hereditary torsion class in $\mathcal{H}_{\theta(C)}$ is determined by its support. 

\begin{example}\label{standard}
If $E$ is an injective cogenerator of $\mathsf{Mod}(R)$, then it is a pure-injective cosilting in $\mathsf{D}(R)$ and $\theta(E)$ is nothing but the standard t-structure. In this case, $\psi_E$ is the constant function sending every prime in $\mathsf{Spec}(R)$ to zero. It is a celebrated theorem of Gabriel (\cite{G}) that the assignment of support considered above establishes a bijection between hereditary torsion classes in $\mathsf{Mod}(R)$ and specialisation-closed subsets of $\mathsf{Spec}(R)$. A recent result of Angeleri Hügel and Hrbek (\cite[Lemma 4.2]{AH1}) shows that a torsion class in $\mathsf{Mod}(R)$ is hereditary if and only if it is of finite type. We can therefore state that the assignment
\begin{equation}\nonumber
\mathfrak{s}\colon \mathsf{Tors}_{\mathsf{Mod}(R)}[\mathsf{Her}]=\mathsf{Tors}_{\mathsf{Mod}(R)}[\mathsf{Her,FT}]\longrightarrow \mathscr{O}^H(\mathsf{Spec}(R))
\end{equation}
sending a hereditary torsion pair of finite type $\mathbf{t}:=(\mathcal{T},\mathcal{F})$ to the specialisation closed subset $\mathsf{supp}(\mathcal{T})$ is a bijection. We will be able to state an analogous result for hearts of certain t-structures later on (see Theorem \ref{der equiv and mut condition}(1)). 
\end{example}

We are now able to carry right mutation, in a compatible way to what was described in the previous subsections, to the sets $\mathsf{sp\mbox{-}filt}^b(R)$ and $\mathsf{Hom}^b_{\mathsf{pos}}(\mathsf{Spec}(R),\mathbb{Z})$. It is more convenient for us to establish this transformation using mutation via torsion pairs.

\begin{theorem}\label{mutation of combinatorial data}
Let $\mathbb{T}$ be a t-structure in $\mathsf{t\mbox{-}str}_{\mathsf{D}(R)}[\mathsf{Cosilt}^*,\mathsf{Int}]=\mathsf{t\mbox{-}str}_{\mathsf{D}(R)}[\mathsf{CG},\mathsf{Int}]$, and consider the associated bounded sp-filtration $\varphi:=\beta(\mathbb{T})$ and bounded homomorphism of posets $\psi:=\mathsf{f}_\varphi$. Take as input $\mathbf{t}$ in $\mathsf{Tors}_{\mathcal{H}_\mathbb{T}}[\mathsf{Her,FT}]$ (i.e. $\mathbf{t}$ satisfies the mutatibility condition), with $\mathsf{supp}(\mathcal{T})=:W$. Consider the output $\mu^{-}_\mathbf{t}(\mathbb{T})$.
\begin{enumerate}
\item The bounded homomorphism of posets $\mu^-_W(\psi):=\gamma \mathcal{H}_{\mu^{-}_\mathbf{t}(\mathbb{T})}$ is defined by
\begin{equation}\nonumber
\mu^-_W(\psi)(\mathfrak{p})=\begin{cases}\psi(\mathfrak{p})+1&{\rm if\ \mathfrak{p}\in W}\\ \psi(\mathfrak{p})&{\rm if\ \mathfrak{p}\notin W}\end{cases}
\end{equation}
\item The sp-filtration $\mu^{-}_W(\varphi):=\beta(\mu^{-}_\mathbf{t}(\mathbb{T}))$ is defined by
\begin{equation}\nonumber
\mu^{-}_W(\varphi)(n)=(W\cap \varphi(n-1))\cup \varphi(n)
\end{equation}
\end{enumerate}
\end{theorem}
\begin{proof}
(1) By the construction of  $\mu^{-}_\mathbf{t}(\mathbb{T})$ we know that $\mathcal{H}_ {\mu^{-}_\mathbf{t}(\mathbb{T})}=\mathcal{F}\ast\mathcal{T}[-1]$. It is known (see \cite[Proposition 4.1]{PV}) that for any given prime ideal $\mathfrak{p}$, the shift of $k(\mathfrak{p})$ that lies in $\mathcal{H}_\mathbb{T}$ either lies in $\mathcal{T}$ (and this happens if and only if $\mathfrak{p}$ lies in $W$) or it lies in $\mathcal{F}$ (and this happens if and only if $\mathfrak{p}$ does not lie in $W$). By definition of the assignment $\gamma$, we obtain the description of $\mu^-_W(\psi)(\mathfrak{p})$ as claimed.

(2) We determine $\mu^{-}_W(\varphi)$ by determining $\mathsf{f}^{-1}(\mu^-_W(\psi))$. As described in Lemma \ref{combinatorics}, $\mu^{-}_W(\varphi)(n)$ can be calculated as $\mu^-_W(\psi)^{-1}(]n,+\infty[)$, i.e. the set of primes $\mathfrak{p}$ for which $\mu^-_W(\psi)(\mathfrak{p})> n$. These are precisely the primes $\mathfrak{p}$ that lie in $W$ and for which $\psi(\mathfrak{p})> n-1$ or the primes $\mathfrak{p}$ that do not lie in $W$ and for which $\psi(\mathfrak{p})> n$. In other words, we have
\begin{align*}
\mu^{-}_W(\varphi)(n)&=(\varphi(n-1)\cap W)\cup (\varphi(n)\cap (\mathsf{Spec}(R)\setminus W))\\
&=(W\cap \varphi(n-1))\cup \varphi(n).\qedhere
\end{align*}
\end{proof}

\begin{remark}
Note that there is a mild discrepancy in the notations $\mu^-_{\mathscr{E}}$, $\mu^-_\mathbf{t}$ and $\mu^-_W$. In the first case, the class in subscript indicates the part of the cosilting object we want mutation to keep; in the second case we write the whole torsion pair in subscript, indicating therefore both the part to be kept (torsionfree class) and the part to change (torsion class); in the third case we indicate the part of $\mathsf{Spec}(R)$ whose values (of a given bounded homomorphism os poset) will be changed. Despite the discrepancy, we hope that these choices allows notation to be more manageable in each setting.
\end{remark}

Note that in order to fully describe the mutation process in $\mathsf{Hom}_{\mathsf{pos}}(\mathsf{Spec}(R),\mathbb{Z})$ or in  $\mathsf{sp\mbox{-}filt}^b(R)$, we need to be able to identify the mutability condition for the subset $W$ of $\mathsf{Spec}(R)$, i.e. to describe the properties of a subset $W$ of $\mathsf{Spec}(R)$ which is the support of a torsion class $\mathcal{T}$ of a torsion pair $\mathbf{t}$ in $\mathsf{Tors}_{\mathcal{H}_\mathbb{T}}[\mathsf{Her,FT}]$, where $\mathbb{T}$ is a t-structures in $\mathsf{t\mbox{-}str}_{\mathsf{D}(R)}[\mathsf{Cosilt}^*,\mathsf{Int}]$. We will see that this can be done under some assumptions on $\mathbb{T}$.

\section{Derived equivalences}\label{der eq section}

\subsection{t-structures inducing derived equivalences} Our current understanding of how mutation acts on all vertices of the diagram of Theorem \ref{diagram} is entangled with the study of t-structures inducing derived equivalences. We say that a t-structure $\mathbb{T}$ in a triangulated category $\mathcal{D}$ \textbf{induces a derived equivalence} if the heart $\mathcal{H}_\mathbb{T}$ of $\mathbb{T}$ is such that the inclusion $\epsilon_\mathbb{T}\colon \mathcal{H}_\mathbb{T}\longrightarrow \mathcal{D}$ extends to a fully faithful triangle functor $F_\mathbb{T}\colon \mathsf{D}^b(\mathcal{H}_\mathbb{T})\longrightarrow \mathcal{D}$. Note that the terminology may be slightly misleading: $F_\mathbb{T}$ does not necessarily need to be an equivalence, but it of course induces an equivalence between $\mathsf{D}^b(\mathcal{H})$ and the essential image of $F_\mathbb{T}$. It is well-known from \cite{BBDG} that if $\mathbb{T}$ is a bounded t-structure and $F_\mathbb{T}$ is faithful, then the essential image of $F_\mathbb{T}$ is $\mathcal{D}$ itself. Note that all t-structures in $\mathsf{t\mbox{-}str}_{\mathsf{D}^b(R)}[\mathsf{Int}]$ are bounded.

\begin{example}
Let us discuss how derived equivalences of rings (derived Morita theory) fit into this framework. Let $R$ be a ring and consider a tilting complex $T$ in the sense of Rickard (\cite{Rick}), i.e., $T$ is a bounded complex of finitely generated projective $R$-modules such that $\mathsf{Hom}_{\mathsf{D}(R)}(T,T[n])=0$ for all $n\neq 0$ and for which the smallest thick subcategory containing $T$ is precisely the category of compact objects in $\mathsf{D}(R)$. Then we know that:
\begin{enumerate}
\item The pair $\mathbb{T}:=(\mathsf{T}^{\perp_{\geq 0}},\mathsf{T}^{\perp_{<0}})$ is a t-structure in $\mathsf{D}(R)$;
\item The heart $\mathcal{H}_\mathbb{T}$ of $\mathbb{T}$ is cocomplete and has a small projective generator, namely $T$ itself. This means that $\mathcal{H}_\mathbb{T}$ is equivalent to the category of right modules over $\mathsf{End}_{\mathsf{D}(R)}(T)$;
\item There is a fully faithful triangle functor $F\colon \mathsf{D}^b(\mathcal{H}_\mathbb{T})\longrightarrow \mathsf{D}(R)$ extending the embedding $\epsilon_{\mathbb{T}}$ of $\mathcal{H}_\mathbb{T}$ into $\mathsf{D}(R)$;
\item The essential image of the functor $F$ is precisely $\mathsf{D}^b(\mathsf{Mod}(R))$ and, thus, $F$ induces a triangle equivalence $\mathsf{D}^b(\mathsf{Mod}(\mathsf{End}_{\mathsf{D}(R)}(T)))\longrightarrow \mathsf{D}^b(\mathsf{Mod}(R))$;
\item It turns out that $F$ can also be extended to a triangle equivalence 
\begin{equation}\nonumber
\hat{F}\colon \mathsf{D}(\mathsf{End}_{\mathsf{D}(R)}(T))\longrightarrow\mathsf{D}(R).
\end{equation}
\end{enumerate}
As a consequence of Rickard's work in \cite{Rick}, whenever two rings $R$ and $S$ have equivalent derived categories, there is an equivalence functor obtained by the (fairly involved) outline described above (based on \cite{BBDG}, see also \cite[Section 6]{AJSS} and \cite{AI,PV}). It is not clear whether such an equivalence functor is uniquely determined by the t-structure. 
\end{example}

In a compactly generated triangulated category, pure-injective cosilting objects are parametrising \textit{nice enough} t-structures with Grothendieck hearts (see Theorem \ref{t-structures prelim}(1)). Therefore, by determining which cosilting t-structures induce derived equivalences, we generalise the derived Morita theory of Rickard discussed above, in the sense that every module category is a Grothendieck category. This generalisation encompasses known examples of derived equivalences that do not fit into Rickard's framework. For example, if we want to discuss how the celebrated derived equivalence shown by Beilinson in \cite{Bei} between quasi-coherent sheaves on the $1$-dimensional projective space and the representations of the Kronecker quiver fits in this framework, we should bring cosilting t-structures into the picture (see for example \cite{S}).

\begin{proposition}\cite[Proposition 5.1]{PV}
Let $\mathcal{\mathcal{G}}$ be a Grothendieck abelian category, $\mathcal{D}(\mathcal{G})$ its derived category, and $C$ a cosilting object in $\mathcal{D}(\mathcal{G})$. Then $\theta(C)$ induces a derived equivalence if and only if $C$ is cotilting. 
\end{proposition}

In this proposition, the assumption on the triangulated category (that it is the derived category of a Grothendieck category) serves the purpose of guaranteeing the existence of a so-called realisation functor, as defined in \cite{BBDG}. The result holds for other triangulated categories that admit an enhancement that is suitable for the construction of such a functor. We refer to \cite{PV} for a detailed discussion of the use of $f$-categories for this purpose (following Beilinson's appendix in \cite{Bei2}) or to \cite{Vi} for a discussion of the (significant) advantages of working in a triangulated category that lies at the base of a stable derivator (and, in fact, $\mathcal{D}(\mathcal{G})$ is one such category).

\subsection{The commutative noetherian case}
We now turn into the case of $\mathsf{D}(R)$ for a commutative noetherian ring $R$. Looking at the diagram of Theorem \ref{diagram}, we consider the subset $\mathsf{Cotilt}^*(\mathsf{D}(R))$ of $\mathsf{Cosilt}^*(\mathsf{D}(R))$. At present, we do not know how to characterise the sp-filtrations or the bounded homomorphisms of posets that are associated to $\mathsf{Cotilt}^*(\mathsf{D}(R))$. Some progress in that direction is made in \cite{HNS}. Nevertheless, what we can say (and we shall in the following theorem) is that the operation $\alpha\circ \mathsf{t\mbox{-}Lift}$ which associates a bounded (pure-injective) cosilting complex to an intermediate t-structure in $\mathsf{D}^b(R)$ factors via $\mathsf{Cotilt}^*(\mathsf{D}(R))$. In other words, this means that every t-structure in the image of $\mathsf{t\mbox{-}Lift}$ induces a derived equivalence. Moreover, for these t-structures we can parametrise precisely the torsion pairs that give rise to right mutations. Recall the assignment
\begin{equation}\nonumber
\hat{\mathfrak{s}}\colon \mathsf{Tors}_{\mathcal{H}_{\theta(C)}}[Her]\longrightarrow 2^{\mathsf{Spec}(R)}
\end{equation}
introduced in Lemma \ref{basics on support for hearts} sending a hereditary torsion pair $\mathbf{t}=(\mathcal{T},\mathcal{F})$ in $\mathcal{H}_{\theta(C)}$ (with $C$ in $\mathsf{Cosilt}^b(\mathsf{D}(R))$) to $\mathsf{supp}(\mathcal{T})$. 

\begin{theorem}\label{der equiv and mut condition}
Let $R$ be a commutative noetherian ring, $\mathbb{S}$ an intermediate t-structure in $\mathsf{D}^b(R)$ and $\mathbb{T}:=\mathsf{t\mbox{-}Lift}(\mathbb{S})$. The following statements hold.
\begin{enumerate}
\item \cite[Theorem 6.16]{PaV} The t-structure $\mathbb{T}$ induces a derived equivalence.
\item\cite[Corollary 6.18]{PaV} The assignment $\hat{\mathfrak{s}}$ restricts to a bijection
\begin{equation}\nonumber 
\mathfrak{s}\colon \mathsf{Tors}_{\mathcal{H}_\mathbb{T}}[\mathsf{Her},\mathsf{FT}]\longrightarrow \mathscr{O}^H(\mathsf{Spec}(R)).
\end{equation}

\end{enumerate}
\end{theorem}
Note that part (2) of the theorem generalises the statement of Gabriel in \cite{G}, reviewed in Example \ref{standard}.
\begin{proof}
 We provide an outline of the steps of the proof, pointing the reader to some of the main ingredients (for further details we refer to \cite{PaV}).
 
 \begin{enumerate}
\item[\underline{Step 1}] \textit{$\mathbb{T}$ is an iterated right mutation of a shift of the standard t-structure, and each hereditary torsion pair of finite type involved in this iteration is in the image of the $\mathsf{Lift}$ assignment in the corresponding heart.} 
\end{enumerate}
Recall from Lemma \ref{basics on support for hearts}(2) that $\mathsf{supp}^{-1}(V)$ is a hereditary torsion class in $\mathcal{H}_\mathbb{T}$, for any $V\subseteq \mathsf{Spec}(R)$ specialisation-closed. Using the classification of compactly generated intermediate t-structures in $\mathsf{D}(R)$ in terms of sp-filtrations, we get a precise recipe of how to iteratively build $\mathbb{T}$ via HRS-tilts at hereditary torsion pairs (since each step of the filtration is specialisation-closed). Since at each step of the iteration we still have a compactly generated t-structure, it follows from \cite[Proposition 6.1]{SSV} that the hereditary torsion pairs we are tilting at are of finite type. Note that this argument, so far, applies to any intermediate compactly generated t-structure. It remains to see that each torsion pair in this iteration restricts to the subcategory of finitely presented objects in the corresponding heart. This can be proved inductively on the length of the (bounded) sp-filtration, by using the fact that $\mathbb{T}=\mathsf{t\mbox{-}Lift}(\mathbb{S})$. 

\begin{enumerate}
\item[\underline{Step 2}] \textit{Assertion (1) holds.} 
\end{enumerate}
We again use an inductive argument on the length of the bounded sp-filtration associated to $\mathbb{T}$. The key for the inductive step is the criterion established in \cite{CHZ} for when an HRS-tilt at a torsion pair induces a derived equivalence. This criterion turns out to be verified precisely because the iterated right mutations occur at torsion pairs that are lifted from torsion pairs in the subcategory of finitely presented objects, as discussed in Step 1.

Consider now the diagram induced by the injective map $\hat{\mathfrak{s}}$ (see Lemma \ref{basics on support for hearts}), 
\begin{equation}\nonumber
\xymatrix{\mathsf{Tors}_{\mathcal{H}_\mathbb{T}}[\mathsf{Her},\mathsf{FT}]\ar@{^{(}->}[d]_{\mathsf{inc}}\ar@{-->}[rr]^{\mathfrak{s}}&&\mathscr{O}^H(\mathsf{Spec}(R))\ar@{^{(}->}[d]_{\mathsf{inc}}\\ \mathsf{Tors}_{\mathcal{H}_\mathbb{T}}[\mathsf{Her}]\ar[rr]^{\hat{\mathfrak{s}}}&&2^{\mathsf{Spec}(R)}}
\end{equation}
where $\mathsf{inc}$ denotes inclusion maps. As recalled above $\mathsf{supp}^{-1}(V)\cap \mathcal{H}_\mathbb{T}$ is a hereditary torsion class in $\mathcal{H}_\mathbb{T}$, for any $V$ in $\mathscr{O}^H(\mathsf{Spec}(R))$. To prove (2) we want the restriction $\mathfrak{s}$ of $\hat{\mathfrak{s}}$ to be well-defined and surjective.

\begin{enumerate}
\item[\underline{Step 3}] \textit{ $\hat{\mathfrak{s}}(\mathsf{Tors}_{\mathcal{H}_\mathbb{T}}[\mathsf{Her},\mathsf{FT}])\subseteq \mathscr{O}^H(\mathsf{Spec}(R))$, and  $\mathfrak{s}$ is well-defined.} 
\end{enumerate}
If $\mathbf{t}:=(\mathcal{T},\mathcal{F})$ lies in $\mathsf{Tors}_{\mathcal{H}_\mathbb{T}}[\mathsf{Her},\mathsf{FT}]$, since $\mathcal{H}_\mathbb{T}$ is a locally coherent category with $\mathsf{fp}(\mathcal{H}_\mathbb{T})=\mathcal{H}_\mathbb{S}$ (see Remark \ref{loc coherent heart}), we can argue that $\mathsf{supp}(\mathcal{T})=\mathsf{supp}(\mathcal{T}\cap \mathcal{H}_\mathbb{S})$. Then we observe that $\mathcal{T}$ and $\mathcal{T}\cap \mathcal{H}_\mathbb{S}$ generate the same localising subcategory in $\mathsf{D}(R)$ and that such localising subcategory must be compactly generated by \cite[Proposition 3.10]{AJS}. The result then follows from Neeman's classification of smashing subcategories (\cite{N}).

\begin{enumerate}
\item[\underline{Step 4}]: \textit{$\mathfrak{s}$ is surjective.} 
\end{enumerate}
It remains to argue that given a specialisation-closed subset $V$ of $\mathsf{Spec}(R)$, the hereditary torsion class $\mathsf{supp}^{-1}(V)\cap \mathcal{H}_\mathbb{T}$ in $\mathcal{H}_\mathbb{T}$ is indeed of finite type. For this purpose we again use the classification of smashing subcategories in terms of specialisation-closed subsets proved by Neeman in \cite{N}, and we are able to draw the desired conclusion using the fact that $\mathbb{T}$ induces a derived equivalence by Statement (1).
\end{proof}

\begin{remark}
Let $\mathbb{S}$ be an intermediate t-structure in $\mathsf{D}^b(R)$. If $\psi\colon \mathsf{Spec}(R)\longrightarrow \mathbb{Z}$ and $\varphi \colon \mathbb{Z}\longrightarrow 2^\mathsf{Spec}(R)$  are, respectively, the associated bounded homomorphism of posets and bounded sp-filtration for $\mathbb{T}:=\mathsf{t\mbox{-}Lift}(\mathbb{S})$, then the mutability condition on a subset $W$ of $\mathsf{Spec}(R)$ is made explicit in part (2) of the theorem: $\psi$ (or $\varphi$) admits a right mutation at $W$ if and only if $W$ lies in $\mathscr{O}^H(\mathsf{Spec}(R))$.
\end{remark}

The fact that $\mathbb{T}=\mathsf{t\mbox{-}Lift}(\mathbb{S})$ induces a derived equivalence ($\mathbb{S}$ intermediate t-structure in $\mathsf{D}^b(R)$) yields two triangulated equivalences, namely $\mathsf{D}^b(\mathcal{H}_\mathbb{T})\longrightarrow \mathsf{D}^b(\mathsf{Mod}(R))$ and its extension $\mathsf{D}(\mathcal{H}_\mathbb{T})\longrightarrow \mathsf{D}(R)$ (see \cite{Vi} for the construction of this functor, an \textit{unbounded realisation functor}). It is natural to ask whether these functors restrict to a triangle equivalence at the level of the bounded derived category of $\mathcal{H}_\mathbb{S}$ and the bounded derived category of finitely generated $R$-modules. The answer is positive and the key for this restriction is to be able to identify $\mathsf{D}^b(\mathcal{H}_\mathbb{S})=\mathsf{D}^b(\mathsf{fp}(\mathcal{H}_\mathbb{T}))$ intrinsically inside $\mathsf{D}^b(\mathcal{H}_\mathbb{T})$, as was done in \cite[Lemma 3.11]{HP} (indeed $\mathsf{D}^b(\mathcal{H}_\mathbb{S})$ turns out to be the subcategory of compact objects in $\mathsf{D}^b(\mathcal{H}_\mathbb{T})$).

\begin{corollary}
Let $R$ be a commutative noetherian ring and $\mathbb{S}$ an intermediate t-structure in $\mathsf{D}^b(R)$ with heart $\mathcal{H}_\mathbb{S}$. Then
\begin{enumerate}
\item \cite[Theorem 3.10]{HP} there is an equivalence $\mathsf{D}^b(\mathcal{H}_\mathbb{S})\longrightarrow \mathsf{D}^b(R)$;
\item \cite[Proposition 6.20]{PaV} the assignment of support induces a bijection between Serre subcategories of $\mathcal{H}_\mathbb{S}$ and specialisation-closed subsets of $\mathsf{Spec}(R)$;
\item $\mathsf{t\mbox{-}Lift}(\mathbb{S})$ is an iterated right mutation of a shift of the standard t-structure.
\end{enumerate}
\end{corollary}
\begin{proof}
Note that (3) follows from Step 1 in the previous proof. We comment briefly on (2). Since we have $\mathcal{H}_\mathbb{S}=\mathsf{fp}(\mathcal{H}_\mathbb{T})$, where $\mathbb{T}:=\mathsf{t\mbox{-}Lift}(\mathbb{S})$, statement (2) follows from Statement (2) of the previous theorem combined with the fact that in a locally coherent Grothendieck category there is always a bijection between Serre subcategories of $\mathcal{H}_\mathbb{S}$ and hereditary torsion classes of finite type in $\mathcal{H}_\mathbb{T}$ (see \cite{Herzog} and \cite{KraLoc}).
\end{proof}

\section{Mutations for bounded t-functions}\label{comm mutation}

In the last section we have seen that we have a combinatorial theory of right mutations ready to be applied to the bounded homomorphisms of posets $\mathsf{Spec}(R)\longrightarrow \mathbb{Z}$ which are associated to (lifts of) intermediate t-structures in $\mathsf{D}^b(R)$. 

\begin{definition}
A function $\psi\colon \mathsf{Spec}(R)\longrightarrow \overline{\mathbb{Z}}$ is said to be a \textbf{t-function} if for a minimal inclusion of prime ideals $\mathfrak{p}\subsetneq \mathfrak{q}$ we have $\psi(\mathfrak{p})\leq \psi(\mathfrak{q})\leq \psi(\mathfrak{p})+1$. A t-function $\psi$ is said to be \textbf{bounded} if the image of $\psi$ is contained in an integer interval $[-n,n]$ for some integer $n$. We denote the set of bounded t-functions $\mathsf{Spec}(R)\longrightarrow \mathbb{Z}$ by $\mathsf{t\mbox{-}Fun}^b(R)$.
\end{definition}

Note that, in particular, $\mathsf{t\mbox{-}Fun}^b(R)\subseteq \mathsf{Hom}_\mathsf{pos}^b(\mathsf{Spec}(R),\mathbb{Z})$ and, it turns out that they often correspond precisely to the t-structures mentioned above. Note also that, for a commutative noetherian ring of finite Krull dimension, any integer valued t-function $\mathsf{Spec}(R)\longrightarrow \mathbb{Z}$ is bounded: once the values of the (finitely many) minimal prime ideals are chosen (and these choices are not in general free from each other), the Krull dimension of $R$ determines an interval in which values of a t-function can lie. 

The following theorem states precisely the correspondence between t-functions and t-structures for a large class of rings: CM-excellent rings of finite Krull dimension. Since the purpose of this article is to illustrate and simplify cosilting mutation for a large class of examples, it will suffice to say that any commutative noetherian ring of finite Krull dimension that is a quotient of a regular (or more generally, of a Cohen-Macaulay) ring is CM-excellent (see \cite{Tak}).
We refer to \cite[Theorem 6.12]{Sta}, as well as \cite[Theorem 6.9]{AJS} for some predecessors of this theorem. 
\begin{theorem}\cite[Theorem 5.5]{Tak}
Let $R$ be a commutative noetherian ring, and suppose $R$ is CM-excellent of finite Krull dimension. The assignment $\mathsf{f}\circ \beta\circ \mathsf{t\mbox{-}Lift}$ determines a bijection between $\mathsf{t\mbox{-}str}_{\mathsf{D}^b(R)}[\mathsf{Int}]$ and $\mathsf{t\mbox{-}Fun}^b(R)$.
\end{theorem}
\begin{remark}
Note that, in comparison to the reference cited, we have added the words \textit{intermediate} and \textit{bounded} to the t-structure side and to the t-function side, respectively. This restriction of the cited bijection is easy to observe, and it can also be seen in Theorem \ref{diagram}, where correspondence between t-structures and sp-filtrations or homomorphisms of posets associates the intermediate condition (on t-structures) to a boundedness condition (on morphisms of posets).
\end{remark}

As a consequence of this theorem and of the last section, we have a complete description of right mutations for t-functions as follows. Let $R$ be a commutative noetherian ring of finite Krull dimension which is CM-excellent and let $\psi\colon\mathsf{Spec}(R)\longrightarrow \mathbb{Z}$ be a t-function. Consider the following procedure that, out of an input subject to a mutability condition, produces an output.
\begin{itemize}
\item \textbf{Input:} Consider a subset $W$ of $\mathsf{Spec}(R)$.
\item \textbf{Mutability condition:} We say that $\psi$ admits a right mutation with respect to $W$ if $W$ is specialisation-closed.
\item \textbf{Output:} Consider the bounded homomorphism of posets $\mu^-_W(\psi)\colon \mathsf{Spec}(R)\longrightarrow \mathbb{Z}$ defined by 
\begin{equation}\nonumber
\mu^-_W(\psi)(\mathfrak{p})=\begin{cases}\psi(\mathfrak{p})+1&{\rm if\ \mathfrak{p}\in W}\\ \psi(\mathfrak{p})&{\rm if\ \mathfrak{p}\notin W}\end{cases}
\end{equation}
\end{itemize}
We saw in Theorems \ref{mutation of combinatorial data} and \ref{der equiv and mut condition} that this does indeed correspond to a mutation of the t-structure associated to $\psi$. Note that, nevertheless, the mutation of a t-function is not necessarily a t-function. This reflects the fact that not all intermediate cosilting t-structures restrict to $\mathsf{D}^b(R)$, despite the fact that all of them are iterated mutations of a shift of the standard t-structure (see Theorem \ref{der equiv and mut condition}).

We end the paper exploring this theory of right mutations in $\mathsf{D}^b(\mathbb{Z})$. Note that $\mathbb{Z}$ is regular (thus CM-excellent) and it has finite Krull dimension (equal to 1).

\begin{example}
Let $\mathbb{P}$ denote the set of prime natural numbers, and recall that, as sets, $\mathsf{Spec}(\mathbb{Z})=\{0\}\cup \mathbb{P}$ (identifying each such integer with the ideal generated by it). Observe that a bounded t-function $\psi\colon \mathsf{Spec}(Z)\longrightarrow \mathbb{Z}$ if completely determined by two pieces of information: $\psi(0)$ and a subset $U=\psi^{-1}(\psi(0)+1)$. It is clear (by definition of a t-function) that $\psi(q)=\psi(0)$ whenever $q$ is in $\mathbb{P}\setminus U$. Hence we have a bijection
\begin{equation}\nonumber
\mathsf{t\mbox{-}Fun}^b(\mathbb{Z})\longrightarrow \mathbb{Z}\times 2^{\mathbb{P}}
\end{equation}
and we will identify a bounded t-function $\psi$ with the pair $(\psi(0),\psi^{-1}(\psi(0)+1))$. 
Note now that the specialisation-closed subsets of $\mathsf{Spec}(\mathbb{Z})$ are 
\begin{equation}\nonumber
\mathscr{O}^H(\mathsf{Spec}(\mathbb{Z}))=\{\mathsf{Spec}(\mathbb{Z})\}\cup 2^{\mathbb{P}}.
\end{equation}
We compute all right mutations of a bounded t-function $(n,U)$ (with $n$ in $\mathbb{Z}$ and $U$ a subset of $\mathbb{P}$) with respect to a specialisation closed subset $W$ of $\mathsf{Spec}(\mathbb{Z})$. First observe that $\mu_{\mathsf{Spec}(\mathbb{Z})}^-(n,U)=(n+1,U)$ since we change all the values of the original t-function. Now, if $W$ is a subset of $\mathbb{P}$, then right mutation of $(n,U)$ with respect to $W$ gives us the following bounded morphism of posets:
\begin{equation}\nonumber
\mu^-_W(n,U)(p)=\begin{cases} n {\rm\ if \ } p\notin W\cup U\\ n+1 {\rm\ if \ } p\in (W\setminus U)\cup (U\setminus W),\\ n+2 {\rm\ if \ } p\in (W\cap U)\end{cases} p\in\{0\}\cup \mathbb{P}.
\end{equation}
From this result, it is worth observing that
\begin{itemize}
\item If $W=\emptyset$, then $\mu^-_\emptyset(n,U)=(n,U)$;
\item If $W\cap U=\emptyset$, then $\mu^-_W(n,U)=(n,U\cup W)$;
\item If $W\cap U\neq \emptyset$, then $\mu^-_W(n,U)$ is no longer a t-function.
\end{itemize}
\end{example}
%
%

\noindent\thanks{\textbf{Acknowledgements:}
The author thanks Lidia Angeleri Hügel, Michal Hrbek, Rosanna Laking, Sergio Pavon and Jan \v{S}\v{t}ov\'i\v{c}ek for inspiring and clarifying discussions, helpful insights and comments on earlier versions of this paper, the Scientific Committee of the International Conference on Representations of Algebras (ICRA) 2022 for the opportunity of presenting this work at the conference, and the editors of this volume for their invitation to write this article. The author acknowledges financial support
from \textit{PRIN - Progetto di Rilevante Interesse Nazionale} through the project \textit{Categories, Algebras: Ring Theoretical and Homological Approaches (CARTHA)} that allowed him to attend ICRA 2022. This work was partially  supported by the Department of Mathematics of the University of Padova via its \textit{BIRD - Budget Integrato per la Ricerca dei Dipartimenti 2022}, through the project \textit{Representations of quivers with commutative coefficients}.} 


\end{document}